\documentclass[smallextended,glov3,final]{svjour3}                     % onecolumn (standard format)
\smartqed  % flush right qed marks, e.g. at end of proof
\usepackage{graphicx}
\usepackage[utf8]{inputenc}
\usepackage[T1]{fontenc}
\usepackage{latexsym}
\usepackage{amsmath}
\usepackage{amssymb}
\usepackage{hyperref}
\usepackage{colortbl}
\usepackage[nosort]{cite}
\newcommand{\R}{\mathbb{R}}
\newcommand{\N}{\mathbb{N}}

\newtheorem{properties}{Properties}

\begin{document}

\title{Positive pseudo almost periodic solutions to a class of hematopoiesis model: Oscillations and Dynamics}

%\titlerunning{Short form of title}        % if too long for running head

\author{ Haifa Ben Fredj \and   Farouk Chérif}

\institute{ Haifa Ben Fredj \at
             MaPSFA, ESSTHS, University of Sousse-Tunisia    \\   
                 Tel.: +21693088956\\      
              \email{haifabenfredjh@gmail.com}  
           \and
           Farouk Chérif \at
          ISSATs and Laboratory of Mathematical Physic, Specials Functions and Applications, LR11ES35, Ecole Supérieure des Sciences et de Technologie de Hammam-Sousse, University of  Sousse, Sousse, Tunisia \\
              Tel.: +21698413089\\
              \email{faroukcheriff@yahoo.fr}           
}

\date{Received: date / Accepted: date}
% The correct dates will be entered by the editor

\maketitle

\begin{abstract}
This paper presents a new generalized Mackey-Glass model with a non-linear harvesting term and mixed delays. The main purpose of this work is to study the existence and the exponential stability of the pseudo almost periodic  solution for the considered model. By using fixed point theorem and under suitable Lyapunov functional, sufficient conditions are given to study the pseudo almost periodic solution for the considered model. Moreover, an illustrative example is given to demonstrate the effectiveness of the obtained results.
% \PACS{PACS code1 \and PACS code2 \and more}
% \subclass{MSC code1 \and MSC code2 \and more}
\end{abstract}

\section{Introduction}
\label{intro}
In 1977, Mackey and Glass \cite{L} proposed the following non-linear differential equation with constant delay
\begin{equation}
x'(t) = -\alpha x(t) +  \dfrac{\beta x(t-\tau)}{\theta^n +  x^n(t-\tau)},\quad 0 <n.
\end{equation}
in order to describe the concentration of mature cells in the blood concentration. Here $\alpha$, $\beta$, $\tau$ and $\theta$ are positive constants, the unknown $x$ stands for the density of mature cells in blood circulation, $\alpha$ is the rate of lost cells from the circulation at time t, the flux \begin{center}
$f(x(t-\tau)):= \dfrac{\beta x(t-\tau)}{\theta^n+ x^n(t-\tau)}$\end{center}
of cells in the circulation depends on $x(t - \tau)$ at the time $t -\tau$, where $\tau$ is the time delay between the
production of immature cells in the bone marrow and their maturation.\\
\\

Since its introduction in the literature, the hematopoiesis model has gained a lot of attention and various extensions.
Hence, under some additional conditions some authors  \cite{K1,K2,M1,N1} considered  an extended version of eq.(1) and obtained  the existence and attractivity of the unique positive periodic and almost periodic solutions  of the following model
 \begin{equation}
 x(t)= {\displaystyle -a(t) x(t) +\sum^{N}_{i=1} \dfrac{b_i(t)x^m(t-\tau_i(t))}{1 + x^n(t-\tau_i(t))}},\quad 0\leq m\leq 1, 0<n.
 \end{equation}

Recently, there have been extensive and valuable contributions dealing with oscillations of  the hematopoiesis model with and without delays, see, e.g., \cite{K1,B,K2,O,H} and references therein.\\
Also, the stability of various models was strongly investigated by many authors recently \cite{M,A,M1,F,B1,N} and references therein.\\
\\
As we all know, in real-world applications equations with a harvesting term provide generally a more realistic and reasonable description for models of mathematical biology and in particular the population dynamics. Hence, the investigation of biological dynamics with harvesting is a meaningful
subject in the exploitation of biological resources which is related to the optimal management of renewable resources \cite{K,P}.\\
\\
Besides, the study oscillations and dynamics systems of biological origin is an exciting topic. One can find a valuable results in this field \cite{H1,H2,B0,M0,M2} and references therein.\\
\\
Motivated by the discussion above the main subject  of this paper is to study the existence and the global attractor
of the unique and positive  pseudo almost periodic solution for the generalized Mackey-Glass model with a nonlinear harvesting term and mixed delays. Roughly speaking, we shall consider the following hematopoiesis model
\begin{eqnarray}
x'(t)=-a(t)x(t)+\sum^{N}_{i=1} \dfrac{b_i(t)x^m(t-\tau_i(t))}{1 + x^n(t-\tau_i(t))}- H(t,x(t-\sigma(t)), \quad  1<m \leq n, t\in \R.
\end{eqnarray}
However, to the author's best knowledge, there are no publications considering the pseudo and positive almost periodic solutions for Mackey-Glass model with harvesting term and $1<m\leq n$.\\
\\
The remainder of this paper is organized as follows: In Section 1, we will introduce some necessary notations, definitions and fundamental properties of the space PAP($\R$,$\R^+$) which will be used in the paper. In Section 2, the model is given. In section 3, 
the existence  of the unique positive pseudo almost periodic solution for the considered system is established. Section 4 is devoted to the stability of the pseudo almost periodic solution. In Section 5, based on suitable Lyapunov function and Dini derivative, we
give some sufficient conditions to ensure that all solutions converge exponentially to the positive pseudo almost periodic solution of the equation (3). At last, an illustrative example is given. It should be mentioned that the main results of this paper are theorems 1, 2.
\section{Preliminaries}
\label{sec:1}
In this section, we would like to recall some basic definitions and lemmas which are used in what follows. In this paper, $BC(\R, \R)$ denotes the set of bounded continued functions from $\R$ to $\R$. Note that $(BC(\R, \R), \|.\|_{\infty})$ is a Banach space where the sup norm $$\|f\|_{\infty} :=\underset{t\in \R}{sup} |f(t)|.$$
\begin{definition} Let u(.) $\in BC(\R, \R)$, u(.) is said to be almost periodic (a.p) on $\R$ if, for any $\epsilon > 0$, the set
$$T(u, \epsilon) = \{\delta; |u(t + \delta) - u(t)| < \epsilon, \text{ for all }t \in \R\}$$ is
relatively dense; that is, for any $\epsilon > 0$, it is possible to find a real number $l = l(\epsilon) > 0$; for any interval with length l($\epsilon$), there exists a number $\delta= \delta(\epsilon)$ in this interval such, that $$|u(t + \delta) - u(t)| <\epsilon, \text{ for all }t \in \R.$$
\end{definition}
We denote by $AP(\R, \R)$ the set of the almost periodic functions from $\R$ to $\R$.
\begin{remark}
Let $u_i (.)$ , $1 \leq i \leq m$ denote almost periodic functions and
$\epsilon > 0$ be an arbitrary real number. Then there exists a positive real number
$L = L(\epsilon)> 0$ such that every interval of length L contains at least one common
$\epsilon-almost$ period of the family of functions $ u_i (.)$, $1 \leq i \leq m.$
\end{remark}
 Besides, the concept of pseudo almost periodicity (p.a.p) was introduced by Zhang \cite{D} in the early nineties. It is a natural generalization of the classical almost periodicity. Precisely, define the class of functions $PAP_0(\R, \R)$ as follows
$$\bigg\{f \in BC(\R, \R); \underset{T\rightarrow+\infty}{lim} \dfrac{1}{2T} \int_{-T}^T |f(t)|dt = 0\bigg\} .$$
A function $f \in BC(\R, \R)$ is called pseudo almost periodic if it can be expressed as $f = h + \phi$, where $h \in AP(\R, \R)$
and $\phi \in PAP_0(\R, \R)$. The collection of such functions will be denoted by $PAP(\R, \R)$. The functions h and $\phi$ in the
above definition are, respectively, called the almost periodic component and the ergodic perturbation of the pseudo almost periodic function $f$. The decomposition given in definition above is unique. It should be mentioned that pseudo almost periodic functions
possess many interesting properties; we shall need only a few of them and for the proofs we shall refer to \cite{D,D1,M3}.
\begin{remark}
Observe that (PAP($\R$, $\R$), $\|.\|_{\infty}$) is a Banach
space and $AP(\R, \R)$ is a proper subspace of $PAP(\R, \R)$ since the function $\psi(t) = cos(2 t) +sin(2 \sqrt{5}t) +exp^{-t^2 |sin (t)|}$ is pseudo almost periodic function but not almost periodic.
\end{remark}
\begin{properties}\cite{D} If f,g $\in PAP(\R,\R)$, then the following assertions hold:
\begin{enumerate}
\item[(a)] f.g, f+g $\in PAP(\R, \R)$.
\item[(b)] $\dfrac{f}{g} \in PAP(\R, \R)$, if $\underset{t\in \R}{inf}|g(t)|>0$.
\end{enumerate}
\end{properties}
\begin{definition}\cite{D}
 Let $\Omega \subseteq \R$ and let K be any compact subset of $\Omega$. On define the class of functions \\$PAP_0(\Omega\times\R, \R)$ as follows
 $$\bigg\{\psi\in C(\Omega\times \R; \R); \underset{T\rightarrow+\infty}{lim} \dfrac{1}{2T} \int_{-T}^T |\psi(s,t)|dt = 0\bigg\}$$
 uniformly with respect to $s\in K$.
\end{definition}
\begin{definition} (Definition 2.12, \cite{E}) \item  Let $\Omega \subseteq \R$. An continuous function f : $\R \times\Omega \longrightarrow \R$  is called pseudo almost periodic (p.a.p).
in t uniformly with respect $x \in \Omega$ if the two following conditions are satisfied :\\
i) $\forall x \in  \Omega, f(., x) \in PAP(\R,\R),$\\
ii) for all compact K of $\Omega$, $\forall \epsilon> 0, \exists \delta > 0, \forall t \in \mathbb{R}, \forall x_1, x_2 \in K$,
\begin{center}
$|x_1 - x_2| \leq \delta \Rightarrow  |f(t, x_1) - f(t, x_2)| \leq \epsilon$.
\end{center}
Denote by $PAP_U(\Omega\times \R; \R)$  the set of all such functions.
\end{definition}

\section{The model}
\label{sec:2}
In order to generalize and improve the above models, let us consider the following Mackey-Glass model with a non-linear harvesting term and several concentrated delays
\begin{eqnarray}
x'(t)=-a(t)x(t)+\sum^{N}_{i=1} \dfrac{b_i(t)x^m(t-\tau_i(t))}{1 + x^n(t-\tau_i(t))}- H(t,x(t-\sigma(t))
\end{eqnarray}
where $t \in \R$ and 
 \begin{enumerate}
 \item[$\blacktriangleright$]The function a : $\mathbb{R}\longrightarrow\mathbb{R^+}$ is pseudo almost periodic(p.a.p) and $\underset{t\in \mathbb{R}}{inf}  a(t) >0.$
 \item[$\blacktriangleright$] For all 1$\leq i \leq N$; the functions $\tau_i,\sigma$, b  : $\mathbb{R}\longrightarrow\mathbb{R^+}$ are p.a.p. 
\item[$\blacktriangleright$] The term H $\in   PAP_U(\mathbb{R}\times\mathbb{R},\R^+$)  satisfies the Lipschitz condition : $\exists L_H > 0$ such that
$${\displaystyle \mid H(t,x)-H(t,y) \mid < L_H \mid x-y \mid,\quad \forall x,y,t \in \mathbb{R}}.$$
\end{enumerate}
 Throughout the rest of this paper, for every bounded function $f : \R \rightarrow \R$, we denote
$$f^+ = \underset {t\in \R}{sup} f(t), f^- =\underset {t\in \R}{inf} f(t).$$ Pose $r =\underset{t \in \mathbb{R}}{sup} \bigg(\tau_i(t),\sigma(t) ; i=1,2...N\bigg).$
 Denote by $BC ([-r, 0] , \R^+$) the set of bounded continuous functions from [-r, 0] to $\R^+$. If $x(.)$ is defined on $[-r + t_0, \sigma[$ with $t_0, \sigma \in \R$, then we define $x_t \in C([-r, 0] , \R$) where $x_t(\theta) = x(t + \theta)$ for all $\theta \in [-r, 0]$. Notice that we restrict our selves to $\R^+$-valued functions since only non-negative solutions of (4) are biologically meaningful. So, let us consider the following initial condition
\begin{equation}
 x_{t_0} = \phi,\quad \phi \in BC ([-r, 0] , \R^+) \text{ and }\phi (0) > 0. 
 \end{equation}
We write $x_t (t_0, \phi)$ for a solution of the admissible initial value problem (4) and (5). Also, let $[t_0, \eta(\phi)[$ be the maximal right-interval of existence of $x_t(t_0, \phi)$.

\section{Main results}
\label{sec:3}
As pointed out in the introduction, we shall give here sufficient conditions which ensures existence and uniqueness of pseudo almost periodic solution of (4). In order to prove this result, we will state the following lemmas. For simplicity, we denote $x_t(t_0, \phi)$ by $x(t)$ for all $t\in \R$.
\begin{proposition}
A positive solution $x(.)$ of model (4)-(5) is bounded on $[t_0, \eta(\phi)[$, and $\eta(\phi)=+\infty$.
\end{proposition}
\begin{proof}\item
 \text{  }   We have for each $t \in [t_0, \eta(\phi)[$ the solution verifies
$$x(t)= {\displaystyle e^{- \int^t_{t_0} a(u) du} \phi(0) +\int^t_{t_0} e^{-\int^t_{s} a(u) du}\bigg[\sum^{N}_{i=1} \dfrac{b_i(s)x^m(s-\tau_i(s))}{1 + x^n(s-\tau_i(s))}- H(t,x(s-\sigma(s))\bigg]ds}.$$
So, by $\underset{x\geq0}{sup}\dfrac{x^m}{1+x^n}\leq 1,\text{ }\forall 1<m\leq n$, we obtain 
$$\begin{array}{lll}
x(t)\leq {\displaystyle \phi(0) +\int^t_{t_0}  e^{-a^-(t-s)}\sum^{N}_{i=1} b_i^+ds}&\leq {\displaystyle \phi(0) + \dfrac{1}{a^-}[1- e^{-a^-(t-t_0)}]\sum^{N}_{i=1} b_i^+}\\
&\leq  \phi(0) + \dfrac{1}{a^-}\sum^{N}_{i=1} b_i^+< +\infty,
\end{array}$$
which proves that $x(.)$ is bounded. The second part of the conclusion is given by Thorem 2.3.1 in \cite{I}, we have that $\eta(\phi)=+\infty$.
\end{proof}

 \begin{proposition}\item
If  $a^->\sum^N_{i=1}b_i^+,$  then  each positive solution $x_t(t_0,\phi)$  of model (4)-(5)  satisfies 
$$x(t) \underset{t \rightarrow +\infty}{\longrightarrow} 0.$$ 
\end{proposition}

\begin{proof}\item
     \text{  }   We define the continuous function
\begin{eqnarray*}G : [0,1]
&\longrightarrow& \R\\
y &\longmapsto& y-a^-+\sum^N_{i=1}b_i^+ e^{y t}.\end{eqnarray*}
From the hypothesis, we obtain G(0)<0, then there exists $\lambda \in [0,1]$, where  $$G(\lambda)<0. \qquad (C.1)$$
 Let us consider the function $W(t)=x(t)e^{\lambda t}$. Calculating the left derivative $W(.)$ and by using the following inequality
$$\dfrac{x^m}{1+x^n} \leq x ,\quad \forall 1<m\leq n. $$We  obtain

  \begin{eqnarray*}
  % \nonumber % Remove numbering (before each equation)
    D^-W(t) &=& \lambda x(t)e^{\lambda t}+x'(t)e^{\lambda t}\\
    \\
    &=&\lambda x(t) e^{\lambda t}+e^{\lambda t}[-a(t)x(t)+\sum_{i=1}^{N} b_i(t) \frac{x^m(t-\tau_i(t))}{1+x^n(t-\tau_i(t))}-H(t,x(t-\sigma(t))]\\
  \\  &\leq& e^{\lambda t}((\lambda- a^-)x(t) + \sum_{i=1}^{N} b^+_i x(t-\tau_i(t))).
  \end{eqnarray*}
  Let us prove that
  \begin{eqnarray*}
  % \nonumber % Remove numbering (before each equation)
    W(t) &=& x(t)e^{\lambda t}< e^{\lambda t_0} M = Q, \forall t\geq t_0.
  \end{eqnarray*}
Suppose that there exists $t_1>t_0$ such that
  \begin{eqnarray*}
  % \nonumber % Remove numbering (before each equation)
    W(t_1) &=& Q, W(t)<Q, \text{ for all } t_0-r\leq t< t_1.
  \end{eqnarray*}
Then
  \begin{eqnarray*}
   0\leq D^-W(t_1)&\leq&(\lambda- a^-)x(t_1)  e^{\lambda t_1}+ \sum_{i=1}^{N} b^+_i x(t_1-\tau_i(t_1))e^{\lambda t_1}  \\
   \\ &\leq& (\lambda- a^-) Q+ \sum_{i=1}^{N} b^+_i  x(t_1-\tau_i(t_1))e^{\lambda t_1} e^{\lambda \tau_i} e^{-(\lambda \tau_i(t_1))}\\
   \\ &=&(\lambda- a^-) Q+ \sum_{i=1}^{N} b^+_i  x(t_1-\tau_i(t_1)) e^{\lambda (t_1-\tau_i(t_1))} e^{\lambda r}\\
    \\&\leq& [\lambda- a^-+ \sum_{i=1}^{N} b^+_i  e^{\lambda r}] Q\\
    \\&<& 0 \qquad (\text{by \textbf{(C.1)}})
  \end{eqnarray*}

which is a contradiction. 
Consequently, $x(t) e^{\lambda t}< Q$. Then, $x(t)< e^{-\lambda t}Q \underset{t \rightarrow +\infty}{\longrightarrow} 0$. 
\end{proof}

\begin{remark} Pose $f_{n,m}(u) =\dfrac{u^m}{1+u^n}$, one can get:\\
\text{\quad  if $m<n$}:
\[
\left  \{ 
\begin{array}{r c l}
f_{n,m}'(u) = \dfrac{u^{m-1}(m-(n-m)u^n)}{(1 + u^n)^2}> 0,\forall u \in \left[ 0,\sqrt[n]{\dfrac{m}{n-m} }\right] \qquad (C.3)\\
\\f_{n,m}'(u) =  \dfrac{u^{m-1}(m-(n- m)u^n)}{(1 + u^n)^2}<0,\forall u \in \left] \sqrt[n]{\dfrac{m}{ n-m} },+\infty \right[ \qquad (C.4),
\end{array}
\right.
 \]\\

 and \text{if } $m= n$:
 $$f_{n,m}'(u) = \dfrac{u^{m-1}m}{(1 + u^m)^2}> 0,\forall u \in [ 0,+\infty[.\qquad  (C.5)$$

If $m<n$, one can choose  $k \in \left]0,  \sqrt[n]{\dfrac{m}{ n-m} }\right[ $ and combining with \textbf{(C.3)} and \textbf{(C.4)}  there exists a constant \\ $\overset{\backsim}{ k}>\sqrt[n]{\dfrac{m}{ n-m} } $ such that 
\begin{center}
$f_{n,m}(k)= f_{n,m}(\overset{\backsim}{k}).$ \qquad (C.6)
\end{center}
Moreover,  $$\underset{u\geq0}{sup}\dfrac{u^m}{1+u^n}\leq 1, \quad \forall 1<m\leq n. \qquad (C.7)$$\\
\end{remark}

  $$\text{Let }\mathcal{C}^0=\{\psi \in BC([-r,0] , \mathbb{R}^+); k\leq \psi \leq M\}.$$
\begin{definition}A positive solution $ x(.)$ of the differential equation is permanent if there exists $t^*\geq 0$, A and B ; $B > A > 0$ such that
$$A \leq x(t) \leq B \quad \text{ for }t \geq t^*.$$
\end{definition}
 \begin{lemma}
 Suppose that there exist a two positives constants M and k satisfying:

 \begin{enumerate}
 \item[[H1]]
 \text{If} m<n:
 $$0<k < \sqrt[n]{\dfrac{m}{ n-m} } < M \leq \overset{\backsim}{k}\quad (\text {$\overset{\backsim}{k}$ was given by \textbf{(C.6)}})$$
 \text{If $ m= n$ we have:}
 $$0<k < M$$

\item[[H2]] ${\displaystyle -a^- M+\sum^N_{i=1}b_i^+ - H^-}<0$
\item[[H3]] ${\displaystyle -a^+ k+\sum^N_{i=1}b_i^- \dfrac{k^m}{1+k^n} -H^+}> 0$
\end{enumerate}
and $\phi \in \mathcal{C}^0$, then the solution of (4)-(5) $ x(.)$ is permanent  which $\eta(\phi)=+\infty$.

\end{lemma}

\begin{proof} \item

\qquad Actually, we prove that $x(.)$ is bounded in $[ t_0,\eta(\phi)[ .$ \\

  $\bullet$ First, we claim that 
  
$$ x(t) < M, \forall  t \in [ t_0,\eta(\phi)[ .\qquad (i)$$

Contrarily, there exists $t_1 \in ] t_0,\eta(\phi)[ $ such that:
\[ 
 \left \{
 \begin{array}{llll}
x(t)<M, \forall t \in  \left[t_0-r,t_1\right[\\
x(t_1) = M
 \end{array} 
 \right. 
 \]

Calculating the right derivative of $x(.)$ and by $\textbf{(H2)}$ and \textbf{(C.7)}, we obtain 
$$
\begin{array}{ll}
0 \leq x'(t_1)&=-a(t_1) x(t_1)+ \sum^{N}_{i=1} \dfrac{b_i(t_1)x^m(t_1-\tau_i(t_1))}{1 + x^n(t_1-\tau_i(t_1))}- H(t_1,x(t_1-\sigma(t_1))\\
\\&\leq  -a(t_1) M +\sum^{N}_{i=1} b_i(t_1) -H^-   \\
\\&< -a^- M +\sum^{N}_{i=1} b_i^+-H^-\\
\\& < 0,
\end{array}
$$
which is a contradiction. So it implies that (i) holds.\\
\\
$\bullet $ Next, we prove that 
\begin{center}
$k < x(t) ,\forall t \in  [ t_0,\eta(\phi)[ .$  \qquad (ii)
\end{center}
Otherwise, there exists $t_2 \in ] t_0,\eta(\phi)[ $ such that 
\[ 
 \left \{
 \begin{array}{llll}
x(t)>k, \forall t \in  \left[t_0-r,t_2\right[\\
x(t_2) = k
 \end{array} 
 \right. 
 \]

 Calculating the right derivative of $x(.)$ and combining with $\textbf{(H3)}$, \textbf{(C.3)} and\textbf{(C.5)}, we obtain 
 $$
\begin{array}{ll}
0 \geq  x'(t_2)&=-a(t_2) x(t_2)+ \sum^{N}_{i=1} \dfrac{b_i(t_2)x^m(t_2-\tau_i(t_2))}{1 + x^n(t_2-\tau_i(t_2))}- H(t_2,x(t_2-\sigma(t_2))\\
\\& \geq -a(t_2) k + \sum^{N}_{i=1} \dfrac{b_i(t_2) k^m}{1 +k^n}-H^+\\
\\& >  -a^+ k + \sum^{N}_{i=1} \dfrac{b_i^- k^m}{1 +k^n}-H^+\\
\\& > 0,\\
\end{array} 
$$
 which is a contradiction and consequnetly (ii) holds. From Thorem 2.3.1 in \cite{I}, we have that $\eta(\phi)=+\infty$. The proof of Lemma 4.4 is now completed.
\end{proof}

$$\text{Let }\mathcal{B}=\{\psi \in PAP(\mathbb{R} , \mathbb{R}); k\leq \psi \leq M\}.$$
 \begin{lemma} 
 $\mathcal{B}$ is a closed subset of $PAP(\R,\R)$.
 \end{lemma}
 \begin{proof}
Let $(\psi_n)_{n \in  \N} \subset \mathcal{B}$ such that $\psi_n \longrightarrow \psi$. 
    $ \text{Let us prove that } \psi \in \mathcal{B}.$\\
    \\
Clearly, $\psi \in PAP(\R,\R)$ and we obtain that  $$\begin{array}{lll} \psi_n \underset{n\longrightarrow +\infty}{\longrightarrow }\psi &\Leftrightarrow \forall \epsilon> 0, \exists n_0>0 \text{ such that }  \mid \psi_n(t)-\psi(t)\mid \leq  \epsilon, (\forall t \in \R, \forall n>n_0)\\&\Leftrightarrow  \forall \epsilon> 0, \exists n_0>0 \text{ such that } -\epsilon \leq \psi_n(t)-\psi(t)\leq\epsilon,(\forall t \in \R, \forall n>n_0).\end{array}$$
Let t $\in \R$, we obtain then
 $$-\epsilon+ k \leq \psi(t)=[\psi(t)-\psi_n(t) ]+\psi_n(t) \leq \epsilon +M.$$
So, $\psi \in \mathcal{B}$.
 \end{proof}
 
\begin{lemma}[Theorem 2.17,\cite{E}]
If f $\in PAP_U(\R \times \R,\R)$ and for each bounded subset B of $\R$, f is bounded on $\R \times B$, then the  Nymetskii operator
$$N_f : PAP(\R,\R) \longrightarrow PAP(\R, \R) \text { with } N_f(u)=f(.,u(.))$$ 
is well defined.
\end{lemma}

\begin{lemma}\cite{G}
Let f,g $\in AP(\R,\R)$. If $g^->0$ then $F \in AP(\R,\R)$ where
$$F(t)={\displaystyle \int^t_{-\infty}e^{-\int^t_s g(u)du} f(s)   ds}, \quad t\in\R.$$
\end{lemma}

\begin{lemma}\cite{D}
Let F $\in PAP_U(\mathbb{R}\times \R,\R)$ verifies the Lipschitz condition: $\exists L_F >0$ such that
$$| F(t,x)-F(t,y)| \leq L_F |x-y|,\quad \forall x,y\in \R \text{ and } t \in \mathbb{R}.$$

If h $\in PAP(\R,\R)$, then the function $F(.,h(.)) \in PAP(\R,\R)$.
\end{lemma}

\begin{theorem}
If conditions  $\textbf{(H1)-( H3)}$ and 
$$\textbf{[H4]}: {\displaystyle \underset{t\in \mathbb{R} }{sup}\bigg\{-a(t)+\sum^N_{i=1}b_i(t)\bigg[\dfrac{(n-m)}{4}+\dfrac{m}{(1+k^n)^2}\bigg]M^{m-1}+L\bigg\}}<0$$ are fulfilled, then the equation (4) has a unique p.a.p solution x(.) in the region $\mathcal{B}$, given by 
$${\displaystyle x(t)=\int^t_{-\infty} e^{-\int^t_s a(u)du}\sum^{N}_{i=1} \dfrac{b_i(s)x^m(s-\tau_i(s))}{1 + x^n(s-\tau_i(s))}- H(s,x(s-\sigma(s))ds}.$$
\end{theorem}

\begin{proof}\item

\text{ }Step 1:\\

Clearly $\mathcal{B}$  is a  bounded set. Now, let  $\psi \in \mathcal{B}$ and $f(t,z)=\psi(t-z)$,  since the numerical application $\psi$ is continuous and the space PAP($\R,\R)$ is a translation invariant then the function $f \in PAP_U(\R \times \R, \R^+)$. Furthermore $\psi$ is bounded, then $f$ is bounded on $\R\times B$. 
By the lemma 4, the Nymetskii operator
\begin{eqnarray*}
 N_f : PAP(\R,\R)&\longrightarrow & PAP(\R,\R)
 \\\tau_i &\longmapsto &f(.,\tau_i(.))\end{eqnarray*}
is well defined for ${\displaystyle \tau_i \in  PAP(\R,\mathbb{R})}$ such that $0 \leq i \leq N$. Consequently,$${\displaystyle \bigg[t\longmapsto\psi(t-\tau_i(t))\bigg]\in  PAP(\R,\mathbb{R^+})}\text{ for all } i=1,...,N.$$

Since ${\displaystyle \underset{t\in \R}{inf} |1+\psi^n(t-\tau_i(t))|>0}$ and using properties 1, the p.a.p functions one has
$${\displaystyle\bigg [ t \longmapsto \sum^{N}_{i=1} \dfrac{b_i(t)\psi^m(t-\tau_i(t))}{1 + \psi^n(t-\tau_i(t))} \bigg] \in PAP(\R,\R)}.$$

Also, under the fact that the  harvesting term verifies the Lipschitz  condition, being the lemma 6,
$${\displaystyle \bigg [G : t \longmapsto \sum^{N}_{i=1} \dfrac{b_i(t)\psi^m(t-\tau_i(t))}{1 + \psi^n(t-\tau_i(t))} -H(t,\psi(t-\sigma(t))\bigg] \in PAP(\R,\mathbb{R})}.$$

\text{ }Step 2:
Let us define the operator $\Gamma$ by
$$\Gamma(\psi)(t)=\int^t_{-\infty}  e^{-\int^t_s a(u)du} G(s)   ds$$
We shall prove that $\Gamma$ maps $\mathcal{B}$ into itself. First, since the functions G(.) and a(.) are  p.a.p one can write $$G=G_1+ G_2 \text{ and }a=a_1+ a_2,$$where $G_1,a_1\in AP(\R,\R)$ and $G_2,a_2 \in PAP_0(\mathbb{R}, \R)$.
So, one can deduce
$$\begin{array}{ll}
\Gamma(\psi)(t)&{\displaystyle=\int^t_{-\infty}  e^{-\int^t_s a_1(u)du} G_1(s)ds+\int^t_{-\infty} \bigg[ e^{-\int^t_s a(u)du} G(s)-e^{-\int^t_s a_1(u)du}G_1(s)\bigg]ds}\\
\\&{\displaystyle=\int^t_{-\infty}  e^{-\int^t_s a_1(u)du} G_1(s)ds+\int^t_{-\infty}  e^{-\int^t_s a_1(u)du} G_1(s)\bigg[e^{-\int^t_s a_0(u)du}-1\bigg]ds}\\
\\&+  \int^t_{-\infty}e^{-\int^t_s a(u)du}G_2(s)ds\\
 \\&{\displaystyle=I(t)+II(t)+III(t). }\end{array}$$\\

 By the lemma \textbf{5},  $I(.) \in AP(\R,\R)$. \\
\\
 Now, we show that II(.) is ergodic. One has
$$\begin{array}{lll}
II(t)&={\displaystyle\int^t_{-\infty}  e^{-\int^t_s a_1(u)du} G_1(s)\bigg[e^{-\int^t_s a_2(u)du}-1\bigg]ds}\\
\\&={\displaystyle \int^{+\infty}_0 e^{-\int^t_{t-v} a_1(u)du} G_1(t-v)\bigg[e^{-\int^t_{t-v} a_2(u)du}-1\bigg]dv}\\
\\&={\displaystyle \int^{v_0}_0 +\int^{+\infty}_{v_0} e^{-\int^{v}_0a_1(t-s)ds} G_1(t-v)\bigg[e^{-\int^{v}_0 a_2(t-s)ds}-1\bigg]dv}\\
\\&=II_1(t)+II_2(t).
\end{array}$$

Since $a^-_1\geq a^-$ and for large enough $v_0$, we obtain 
$$\begin{array}{lll}
|II_2(t)|&\leq {\displaystyle \int^{+\infty}_{v_0} e^{-\int^{v}_0 a_1(t-s)ds} |G_1(t-v)|\bigg[|e^{-\int^{v}_0 a_2(t-s)ds}|+1\bigg]dv}\\
\\&= {\displaystyle \int^{+\infty}_{v_0} \|G_1\|_{\infty}\bigg[e^{-\int^{v}_0 a(t-s)ds}+ e^{-\int^{v}_0 a_1(t-s)ds}\bigg]dv}\\
\\&\leq  {\displaystyle \int^{+\infty}_{v_0} 2 \|G_1\|_{\infty}e^{- a^-v}dv <\dfrac{\epsilon}{2}}.
\end{array}$$
So, $II_2 \in PAP_0(\R,\R)$.\\

Thereafter, it has not yet been demonstrated that $II_1(.) \in PAP_0(\R,\R)$.\\
Firstly,  we prove that the following function
$$\mu(v,t)={\displaystyle \int^v_0 a_2(t-s)ds, \qquad (v\in [0,v_0], t\in \R)}$$
is in $PAP_0([0,v_0] \times \R, \R)$. Clearly $|\mu(v,t)|\leq {\displaystyle\int^v_0 |a_2(t-s)|ds}$, then  it is obviously sufficient to prove that the function $${\displaystyle\int^v_0 |a_2(.-s)|ds \in PAP_0(\R,\R).}$$

We have $a_2 (.)\in PAP_0(\R,\R)$, for $\epsilon >0$ there exists $T_0 >0$ such that
$${\displaystyle \dfrac{1}{2T} \int^T_{-T} |a_2(t-s)|dt \leq \dfrac{\epsilon}{v}, \qquad (T\geq T_0, s\in [0,v])}.$$
Since $[0,v]$ is bounded, the Fubini's theorem gives for $\epsilon >0$ there exists $T_0 >0$ such that $${\displaystyle \dfrac{1}{2T} \int^T_{-T} \int^v_0 |a_2(t-s)|ds dt\leq \epsilon, \qquad T\geq T_0}.$$
So,  $${\displaystyle\int^v_0 |a_2(.-s)|ds \in PAP_0(\R,\R)}$$ which implies the required result.\\

Then, we obtain 
$$\begin{array}{	lll}
{\displaystyle\dfrac{1}{2T} \int^T_{-T}|II_1(t)|dt } & {\displaystyle\leq \dfrac{1}{2T} \int^T_{-T}\int^{v_0}_0 e^{-\int^{v}_0 a_1(t-s)ds} |G_1(t-v)||e^{-\int^{v}_0 a_2(t-s)ds}-1|dv dt }\\
\\& ={\displaystyle\dfrac{1}{2T} \int^T_{-T}\int^{v_0}_0 |G_1(t-v)||e^{-\int^{v}_0 a(t-s)ds}-e^{-\int^{v}_0 a_1(t-s)ds}|dv dt }.
\end{array}$$
By the mean value theorem, $\exists \eta \in ]0,1[$ such that 
\begin{eqnarray*}
{\displaystyle\dfrac{1}{2T} \int^T_{-T}|II_1(t)|dt }\leq &&{\displaystyle \dfrac{1}{2T} \int^T_{-T}\int^{v_0}_0 |G_1(t-v)|e^{-[(1-\eta)\int^{v}_0 a(t-s)ds+ \eta \int^{v}_0 a_1(t-s)ds]}}\\
\\&&\times\bigg(\int^{v}_0|a_2(t-s)|ds\bigg) dv dt .
\end{eqnarray*}

Since the function $\mu \in PAP_U([0,x_0] \times \R, \R)$ and in virtue of  the  Fubini's theorem for $\epsilon>0$, $\exists T_1>0$ such that 
$$\begin{array}{lll}
{\displaystyle\dfrac{1}{2T} \int^T_{-T}|II_1(t)|dt }&{\displaystyle\leq \int^{v_0}_0 \|G_1\|_{\infty} \dfrac{\epsilon}{ 2\|G_1\|_{\infty} v_0}}=\dfrac{\epsilon}{2},  \qquad (\forall T\geq T_1).
\end{array}$$\\
So, $II_1\in PAP_0(\R,\R)$.\\

Finally, we study the ergodicity of III(.). We have 
 
 $$\begin{array}{ll}
{\displaystyle \dfrac{1}{2T}\int^T_{-T} |III(t)|dt}& {\displaystyle \leq \dfrac{1}{2T} \int^T_{-T}  \int ^t _{-\infty} e^{-(t-s) a^-}   \mid G_2(s)  \mid ds dt}\\
\\&\leq III_1(T) + III_2(T),
\end{array}$$
where $$III_1(T)={\displaystyle  \dfrac{1}{2T} \int^T_{-T}  \int ^t _{-T} e^{-(t-s) a^-}    \mid G_2(s)  \mid ds dt}$$ and
$$III_2(T)={\displaystyle \dfrac{1}{2T} \int^{-T}_{-\infty}  \int ^t _{-\infty} e^{-(t-s) a^-}   \mid G_2(s)  \mid ds dt}.$$

Next, we prove that $$\underset{T\rightarrow +\infty}{lim} III_1(T)=\underset{T\rightarrow +\infty}{lim}III_2(T)=0.$$
By the Fubini's theorem, we obtain 
$$\begin{array}{ll}
III_1(T)&{\displaystyle = \int ^{+\infty}_0 e^{- a^- u} \dfrac{1}{2T} \int^T_{-T}  \mid  G_2(t-u)  \mid  dt du}\\
\\&{\displaystyle= \int ^{+\infty}_0 e^{- a^- u} \dfrac{1}{2T} \int^{T-u}_{-T-u}  \mid  G_2(t)  \mid  dt du}\\
\\&{\displaystyle \leq \int ^{+\infty}_0 e^{- a^- u} \dfrac{1}{2T} \int^{T+u}_{-(T+u)}  \mid  G_2(t)  \mid  dt du.}
\end{array} $$

Now, since  $G_2 \in PAP_0(\mathbb{R},\mathbb{R})$, then the function $\Psi_T$  defined by
$${\displaystyle \Psi_T(u)=\dfrac{T+u}{T} \dfrac{1}{2(T+u)} \int^{T+u}_{-(T+u)}  \mid  G_2(t)  \mid  dt} $$ 

is bounded and satisfy $\underset{T\longrightarrow +\infty}{lim} \Psi_T(u)=0$. From the Lebesgue dominated convergence theorem, we obtain $$\underset{T\rightarrow +\infty}{lim}III_1(T)=0.$$ 

On the other hand, notice that $\|G_2\|_\infty<0$ and by setting $\xi=t-s$ we obtain 
$$\begin{array}{ll}
III_2(T)&{\displaystyle \leq \dfrac{\|G_2\|_{\infty} }{2T} \int^T_{-T} \int ^{+\infty}_{2T} e^{- a^- \xi}  d\xi  dt}\\
\\&{\displaystyle =\dfrac{\|G_2\|_{\infty} }{a^-} e^{-2 a^- T}}\qquad \underset{T\longrightarrow +\infty }{\longrightarrow 0}.
\end{array} $$
which implies the required result.\\
\\

Step 3:\\
Let
\vspace{-0.5cm}
\begin{eqnarray*}\gamma : [0,1]&\longrightarrow& \mathbb{R}\\
u&\longmapsto& {\displaystyle \underset{t\in \mathbb{R} }{sup}\bigg\{-a(t)+\bigg[\sum^N_{i=1}b_i(t) \bigg(\dfrac{(n-m)}{4}+\dfrac{m}{(1+k^n)^2}\bigg)M^{m-1}+L\bigg]e^u\bigg\}.}\end{eqnarray*}
It is clear that $\gamma$ is continuous function on [0,1].\\
From $\textbf{(H4)}$ : $\gamma$(0)<0, so $\exists \zeta \in [0,1]$ such that $$\gamma(\zeta)<0\qquad (C.8).$$

Next, we claim that $ \Gamma(\psi)(t) \in [k,M]$ for all $t\in \R .$\\
For $\psi \in \mathcal{B}$, we have $$
\begin{array}{lll}
\bullet \text{ }\Gamma(\psi)(t)&\leq{\displaystyle  \int ^t_{-\infty}  e^{-(t-s) a^-} \left[  \sum^{N}_{i=1} b_i^+ -H^-\right]ds}  \qquad(\text{By \textbf{(C.1)}})\\
     \\&\leq{\displaystyle  \int ^t_{-\infty}  e^{a^-(t-s)} a^- M   ds}     \qquad   ({\text{By \textbf{(H2)}} )}\\
      \\&=M.\\
     \\ \bullet \text{ }\Gamma(\psi)(t)&\geq {\displaystyle  \int^t_{-\infty}  e^{-(t-s)a^+}\left[ \sum^{N}_{i=1} \dfrac{b_i^- k^m}{1+k^n} -H^+\right] ds}  \quad (\text{By \textbf{(C.3)} and \textbf{(C.5)}})\\
 \\ &\geq {\displaystyle \int ^t_{-\infty}  e^{-(t-s)a^+} a^+ k   ds}  \qquad  (\text{By \textbf{(H3)}}  )\\
  \\&=k.
 \end{array}
 $$
Thus $\Gamma$  a self-mapping from $\mathcal{B}$ to $\mathcal{B}$. \\
\\

 $\ast$ $\Gamma$ is a contraction. Indeed;
Let $\varphi ,\psi \in \mathcal{B} $, we have \begin{eqnarray*}
 \Vert \Gamma(\varphi)-\Gamma(\psi)\Vert _\infty &=&\underset{t\in \R}{sup}|\Gamma(\phi)(t)-\Gamma(\psi)(t)|\\
\\ &\leq &{\displaystyle \underset{t\in \mathbb{R}}{sup} \int^t_{-\infty} e^{-\int^t_s a(u)du} \sum^{N}_{i=1} b_i(s) \bigg| \dfrac{\varphi^m(s-\tau_i(s))}{1 + \varphi^n(s-\tau_i(s))}-\dfrac{\psi^m(s-\tau_i(s))}{1 + \psi^n(s-\tau_i(s))}\bigg| }\\
\\&&+\bigg|H(s,\varphi(s-\sigma(s))-H(s,\psi(s-\sigma(s))\bigg|ds.
\end{eqnarray*}
By the mean value theorem, one can obtain $$\begin{array}{lll}
\bigg|\dfrac{x^m}{1+x^n}- \dfrac{y^m}{1+y^n}\bigg|&=|g'(\theta)| |x-y| \qquad \qquad\text{\qquad \text{\qquad }where $\bigg[g : t\in \R^+ \longrightarrow \dfrac{t^m}{1+t^n}\bigg]$}\\ \\
&=\bigg|\dfrac{\theta^{m-1+n}(m-n)+m\theta^{m-1}}{(1 + \theta^n)^2}\bigg| |x-y|\\
\\&\leq \bigg[\dfrac{\theta^{m-1}(n-m)}{4 } +\dfrac{m\theta^{m-1}}{(1+\theta^n)^2}\bigg]|x-y|,\\
\end{array} $$ where $x,y \in [k,M]$, $\theta$ lies between $x$ and $y$.  \\

Consequently, the following estimate hold 
$$
\begin{array}{lll}
{\displaystyle \Vert \Gamma(\varphi)-\Gamma(\psi)\Vert_\infty}& \leq {\displaystyle \underset{t\in \mathbb{R}}{sup} \int^t_{-\infty} e^{-\int^t_s a(u)du} \bigg(\sum^{N}_{i=1} b_i(s) \bigg[ \dfrac{(n-m)M^{m-1}}{4}+\dfrac{m M^{m-1} }{(1+k^n)^2}\bigg]}\\
\\& \bigg| \varphi(s-\tau_i(s))-\psi(s-\tau_i(s))\bigg| +L \parallel \varphi-\psi \parallel_\infty\bigg) ds\\
\\& {\displaystyle  \leq \underset{t\in \mathbb{R}}{sup} \int^t_{-\infty} e^{-\int^t_s a(u)du} \bigg(\sum^{N}_{i=1} b_i(s) \bigg[\dfrac{(n-m)}{4}+\dfrac{m}{(1+k^n)^2}\bigg]M^{m-1}+L\bigg)}\\
\\& {\displaystyle \times \parallel \varphi-\psi \parallel _\infty ds}\\
\\&\leq {\displaystyle  \parallel \varphi-\psi \parallel _\infty \underset{t \in \mathbb{R}}{sup} \int ^t_{-\infty}  e^{-\int^t_s a(u)du}a(s) e^{-\zeta}ds} \qquad (\text{By \textbf{(C.8)}})\\
\\& \leq {\displaystyle e^{-\zeta} \parallel \varphi-\psi \parallel_\infty},\end{array}$$

which proves that $\Gamma$ is a contracting operator on ${\displaystyle \mathcal{B}}$. By using fixed point theorem, we obtain that operator $\Gamma$ has a unique fixed point ${\displaystyle x^* (.)\in \mathcal{B}}$, which corresponds to unique p.a.p solution of the equation (4).
 \end{proof}
\section{The stability of the pap solution}
\label{sec:4}
\begin{definition}\cite{I} Let $f : \R \longrightarrow \R$ be a continuous function, then the upper right derivative of $f$ is defined as
$$D^+f(t)= \overline{\underset{h\rightarrow0^+}{lim}}\dfrac{f(t + h) - f(t)}{h}.$$
\end{definition}
\begin{definition} We say that a solution $x^*$ of Eq. (4) is a global attractor or globally asymptotically stable (GAS) if for any
positive solution $x_t(t_0,\phi)$ 
$$\underset{t\rightarrow 	+\infty}{lim}|x^*(t)-x_t(t_0,\phi)|=0.$$
\end{definition}
\begin{theorem}
Under the assumptions \textbf{H(1)-H(4)}, the positive pseudo almost periodic solution $x^*(.)$ of the equation (4) is  a global attractor.
\end{theorem}

\begin{proof}
Firstly, set $x_t(t_0,\phi)$ for $\phi \in \mathcal{C}^0$ by $x(t)$ for all $t \in \R$. Let  $$y(.)=x(.)-x^*(.).$$Then,
\begin{eqnarray*}
 y'(t) &=& -a(t)[x(t)-x^*(t)]+\sum^{N}_{i=1} b_i(t)\bigg[\dfrac{x^m(t-\tau_i(t))}{1 + x^n(t-\tau_i(t))}-\dfrac{x^*{^m}(t-\tau_i(t))}{1 + x^*{^n}(t-\tau_i(t))}\bigg] \\
\\ &&-\bigg [H(t,x(t-\sigma(t)))-H(t,x^*(t-\sigma(t)))\bigg].
 \end{eqnarray*}
Let us define a continuous function $\Delta :\mathbb{R^+}\longrightarrow \mathbb{R}$ by
$$\Delta(u)= u-a^-+\bigg[L_{H}+\sum_{i=1}^{N} b^+_i\bigg(\dfrac{(n-m)}{4}+\dfrac{m}{(1+k^n)^2}\bigg)M^{m-1}\bigg]\exp(ut).$$

From $\textbf{(H4)}$, we have $\Delta	(0)<0$, then there exists $\lambda\in \mathbb{R}^+$, such that$$\Delta(\lambda)<0 \qquad (C.11).$$

We consider the Lyapunov  functional $V(t)=|y(t)|e^{\lambda t}$. Calculating the upper right derivative of $V(t)$, we obtain
  \begin{eqnarray*}
    D^+V(t) &\leq & \bigg[-a(t)|y(t)|+\sum_{i=1}^{N} b_i(t)\bigg|\frac{x^m(t-\tau_i(t))}{1+x^n(t-\tau_i(t))} -\frac{(x^*)^m(t-\tau_i(t))}{1+(x^*)^n(t-\tau_i(t))}\bigg|+ \bigg|H(t,x^*(t-\sigma(t)))\\
    \\&&-H(t,x(t-\sigma(t)))\bigg|\bigg]e^{\lambda t}+\lambda|y(t)|e^{\lambda t}\\
   \\&\leq & e^{\lambda t}\bigg((\lambda-a^-)|y(t)|+L_H|y(t-\sigma(t))|+\sum_{i=1}^{N} b^+_i \bigg|\frac{x^m(t-\tau_i(t))}{1+x^n(t-\tau_i(t))}-\frac{(x^*)^m(t-\tau_i(t))}{1+(x^*)^n(t-\tau_i(t))}\bigg|\bigg).
  \end{eqnarray*}
We claim that
  \begin{eqnarray*}
  % \nonumber % Remove numbering (before each equation)
    V(t) &=& |y(t)|e^{\lambda t}< e^{\lambda t_0}(M+\max_{t<t_0}|x(t)-x^*(t)|)=M_1, \forall t\geq t_0.
  \end{eqnarray*}
Suppose that there exists $t_1>t_0$ such that
  \begin{eqnarray*}
  % \nonumber % Remove numbering (before each equation)
    V(t_1) &=& M_1, V(t)<M_1, \forall \text{ } t_0-r\leq t< t_1.
  \end{eqnarray*}
Besides,
  \begin{eqnarray*}
  % \nonumber % Remove numbering (before each equation)
   0\leq D^+V(t_1)&\leq& (\lambda-a^-)|y(t_1)|e^{\lambda t_1}+L_H|y(t_1-\sigma(t_1))|e^{\lambda (t_1-\sigma(t_1))}e^{\sigma(t_1) \lambda}\\   
    \\&&+\sum_{i=1}^{N} {b_i}^+ e^{\lambda t_1}\bigg[\frac{x^m(t_1-\tau_i(t_1))}{1+x^n(t_1-\tau_i(t_1))}-\frac{(x^*)^m(t_1-\tau_i(t_1))}{1+(x^*)^n(t_1-\tau_i(t_1))}\bigg].
  \end{eqnarray*}

On the other hand, for all $x,x^* \in \R^+$ we have 

$$\bigg|\dfrac{x^m}{1+x^n}-\dfrac{(x^*)^m}{1+(x^*)^n}\bigg| \leq \bigg[\dfrac{\zeta^{m-1}(n-m)}{4 } +\dfrac{m\zeta^{m-1}}{(1+\zeta^n)^2}\bigg] |x-x^*|,$$

where $\zeta \in [k,M]$.
 \\
 
So, we obtain
 \begin{eqnarray*}
 % \nonumber % Remove numbering (before each equation)
    \bigg |\dfrac{x^m(t-\tau_i(t))}{1+x^n(t-\tau_i(t))}-\frac{(x^*)^m(t-\tau_i(t))}{1+(x^*)^n(t-\tau_i(t))}\bigg | \leq  \bigg[\dfrac{M^{m-1}(n-m)}{4 } +\dfrac{M^{m-1}}{(1+k^n)^2}\bigg]   |y(t-\tau_i(t))|.
   \end{eqnarray*}
   
 Then,
\begin{eqnarray*}
  % \nonumber % Remove numbering (before each equation)
   0&\leq& D^+V(t_1)\\
    \\&\leq& (\lambda-a^-)|y(t_1)|e^{\lambda t_1}+L_H|y(t_1-\sigma(t_1))|e^{\lambda (t_1-\sigma(t_1))}e^{\sigma(t_1)\lambda}\\
    \\ &&+\sum_{i=1}^{N} \bar{b_i}e^{\lambda (t_1-\tau_i(t_1)) }e^{\tau_i(t_1) \lambda}|y(t_1-\tau_i(t_1))|\bigg[\dfrac{M^{m-1}(n-m)}{4 } +\dfrac{m M^{m-1}}{(1+k^n)^2}\bigg]\\
      \\&\leq& \bigg(\lambda-a^-+L_H e^{r\lambda}+\sum_{i=1}^{N} b^+_i e^{r \lambda}\bigg[ \dfrac{(n-m)}{4}+\dfrac{m}{(1+k^n)^2}\bigg]M^{m-1} \bigg)M_1.
  \end{eqnarray*}
  
However, by \textbf{(C.11)}
  \begin{eqnarray*}
  % \nonumber % Remove numbering (before each equation)
  \lambda-a^-+L_H e^{r\lambda}+\sum_{i=1}^{N} b^+_i e^{r \lambda}\bigg [ \dfrac{(n-m)}{4}+\dfrac{m}{(1+k^n)^2}\bigg]M^{m-1} <0,
  \end{eqnarray*}
  which is contradicts the hypothesis. Consequently, $|y(t)|< M_1 e^{-\lambda t},\text{ } \forall t> t_0$.
\end{proof}

 \section{An example}
 \label{sec:5}
 In this section, we present an example to check the validity of our theoretical results obtained in the previous sections.\\
 \\
First, we construct a function $\omega(t)$. For $n = 1, 2,...$ and $0 \leq i < n$, $$a_n=\dfrac{n^3-n}{3}$$ and intervals $$I_n^i = [a_n + i, a_n + i + 1].$$ Choose a non-negative, continuous function g on [0,1] defined by $$g(s) =\dfrac{8}{\pi}\sqrt{s-s^2}.$$
Define the function $\omega$ on $\R$ by 
 \[\omega(t)=
\left \{
\begin{array}{lll}
g[t-(a_n+i)],  &\qquad t \in I_n^i,\\

0,          &\qquad t\in \R^+ \  \cup \{I_n^i: n=1,2,...,0\leq i \leq n \}, \\

\omega(-t), &\qquad t<0.

\end{array}
\right.
\]\\

From {\color{blue}{(\cite{D}, p.211, Example 1.7)}}, we know that $\omega \in  PAP_0( \R, \R)$ is ergodic. However, $\omega \notin C_0( \R, \R)$. 

\begin{figure}[hbtp]
\includegraphics[scale=0.4]{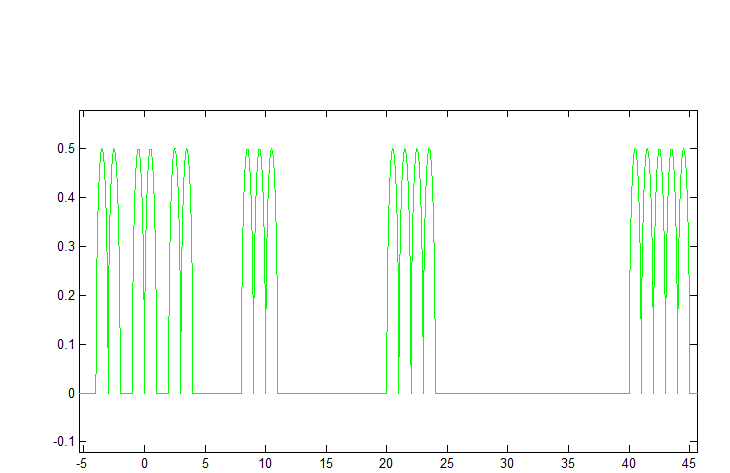}
\caption{Diagram of $\omega$}
\end{figure}

Let us consider the case $n=m=2$ , $N=1$,\\
\\
 $a(t)=0.38+ \dfrac{|sin(t)+sin (\pi t)|}{400} + \dfrac{ \pi\omega(t)}{800}$, $b_1(t) =1+\dfrac { sin^2( t) + sin^2(\sqrt{2}t)}{10}+ \dfrac{1}{100(1+t^2)},$ \\
 \\
 $ \tau_1(t) = cos^2(t)  + cos^2 (\sqrt{2}t) + 1 +e^{-t^2}, \text{}H(t,x)=0.01|sin(t)+cos(\sqrt{3}t)| \dfrac{|x|}{1+x^2},\text{ et }\sigma(t)=|sin(t)-sin(\pi t)|.$\\
 \\
Clearly, $a^+=0.39, a^-=0.38,  b^+=1.21, b^-=1, H^+=5*10^{-3}, H^-=0, r=4.$\\
\\
It is not difficult to see that  H $\in PAP_U(\R \times \R,\R^+)$ and satisfies Lipschitz condition with   $l=10^{-2}$.\\
\\Let k=2, M=3.29. We obtain easily:
\begin{enumerate}
\item[[H1]] $0<2<3.29;$
\item[[H2]] ${\displaystyle -a^- M+\sum^N_{i=1}b_i^+ - H^-}= -0.04 < 0$;
\item[[H3]] ${\displaystyle -a^+ k+\sum^N_{i=1}b_i^- \dfrac{k^m}{1+k^n} -H^+}= 0.015> 0$;
\item[[H4]] ${\displaystyle \underset{t\in \mathbb{R} }{sup}\bigg\{-a(t)+\sum^N_{i=1}b_i(t)\bigg[\dfrac{(n-m)}{4}+\dfrac{m}{(1+k^n)^2}\bigg]M^{m-1}+L\bigg\}}$=   -0.0515<0.
\end{enumerate}

Then, all the conditions in Theorem  {\color{blue}1} et {\color{blue}2} are satisfied, Therefore, there exists a unique  pseudo almost periodic solution $x^*$ in $\mathcal{B}=\{\phi \in PAP(\R,\R); k\leq \phi (t) \leq M,\text{} \forall t\in \R\}$ which is  global attractor. \\
\begin{figure}[hbtp]
\includegraphics[scale=0.7]{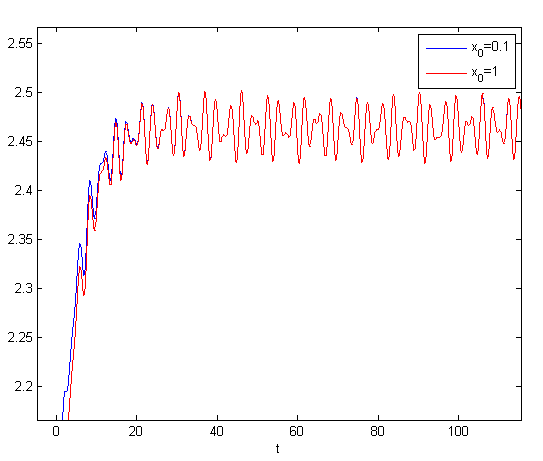}
\caption{The numerical solution  $x_t(t_0,2)$ of the example 6.1 for $x_0=0.1$ and $x_0=1$}
\end{figure}
\begin{remark}
Notice that in vue of this above example, it follows that the condition of proposition 4.2 is necessary. Besides, The results are verified by the numerical simulations in fig(2).
\end{remark}
\section{Conclusion}
\label{sec:6}
In this paper, some new conditions were given ensuring the existence of the uniqueness positive pseudo almost periodic solution of the hematopoies model with mixed delays and with a non-linear harvesting term (which is more realistic).\\
Also, the global attractivity of the unique pseudo almost periodic solution of the considered model is demonstrated by a new and suitable Lyapunov function.\\
\\As the best of our knowledge, this is the first paper considering such solutions for this generalized model.\\
Notice that our approach can be applied to the case of the almost automorphic  and pseudo almost automorphic solutions of the considered model.

\end{document}